\documentclass[12pt, 14paper,reqno]{amsart}
\usepackage{amsmath,amsfonts,amssymb}
\usepackage{mathrsfs}
\usepackage{longtable}
\usepackage[breaklinks]{hyperref}
\usepackage{graphicx}

\theoremstyle{plain}
\newtheorem{thm}{Theorem}[section]
\newtheorem{lem}{Lemma}[section]

\newtheorem{prop}{Proposition}[section]
\newtheorem{q}{Question}[section]

\newtheorem{thma}{Theorem}

\theoremstyle{proof}
\numberwithin{equation}{section}

\begin{document} 
\title[Rank zero quadratic twists of a Mordell curve]{A note on rank zero quadratic twists of a Mordell curve}
\author{Ankurjyoti Chutia, Azizul Hoque and Jyotishman Kalita}
\address{A. Chutia @Department of Mathematics, Gauhati University, Guwahati-781014, Assam, India}
\email{ankurjyoti878@gmail.com}
\address{A. Hoque @Department of Mathematics, Faculty of Science, Rangapara College, Rangapara-784505, Assam, India}
\email{ahoque.ms@gmail.com}
\address{J. Kalita @Department of Mathematics, Gauhati University, Guwahati-781014, Assam, India}
\email{jyotishmankalita2016@gmail.com}
\date{\today}
\maketitle

\begin{abstract}
We produce two families of rank zero quadratic twists of the Mordell curve $y^2=x^3+2$. %Further, we prove that there are infinitely many members in each of these families. %The proof is based on the Scholz reflection principle and basic of the fundamental unit in a real quadratic field.  
At the end, we give numerical examples supporting the result.
\end{abstract}

\section{Introduction}
Let $E$ be an elliptic curve defined over $\mathbb{Q}$ given by the Weierstrass equation,
$$E:~y^2=x^3+ax+b,$$
where $a$ and $b$ are integers with $4a^3+27b^2\ne 0$. We denote the Mordell-Weil group of $E$ by $E(\mathbb{Q})$, and its rank by $r(E)$. Given a square-free integer $D$, the quadratic twist of $E$ by $D$ is an elliptic curve given by the following equation: 
$$E_D:~ y^2 = x^3 + aD^2x + bD^3.$$ 
One can naturally ask the following question:
\begin{q}\label{q1}
What can be said about the behaviour of the ranks of $E_D$ as $D$ varies over the square-free integers?
\end{q}

It is widely believed that ``almost all" elliptic curves over $\mathbb{Q}$ have rank $0$ or $1$. Further, Goldfeld \cite{Go79} conjectured that the average rank of the quadratic twists of any given elliptic curve over $\mathbb{Q}$ is $1/2$.  Consequently, for any elliptic curve over $\mathbb{Q}$, asymptotically, there are at least half of the quadratic twists of this curve which have rank $0$.  Thus, there is a comparatively weaker conjecture stating that there is a positive proportion of all quadratic twists of any elliptic curve over $\mathbb{Q}$ which have rank $0$. In the general case, this conjecture, though much weaker than the other famous ones related to elliptic curves, is still open. There have been numerous papers treating this problem for modular curves. Most of them focused on the non-vanishing of the $L$-functions (see \cite{Ko88, Ko99, OS98, Yu05}). Iwaniec and Sarnak \cite{IS00}, under the Riemann hypothesis, proved that half of the quadratic twists of an elliptic curve over $\mathbb{Q}$ have rank $0$. Unconditionally, such positive proportion results are only known for a few elliptic curves. 

In this article, we prove a result to enlist to the aforementioned literature on rank zero quadratic twists of a Mordell curve. We consider the family of Mordell curves $y^2=x^3+k$ with $k$ runs over the  non-zero square-free integers. This family of Mordell curves has been well studied, and its various arithmetic properties  have also been explored (see for instance, \cite{CH, KI, Qi16}) over the years. Similar studies (see, \cite{MM, ON, WQ}) have been done for the families $E_D:~ y^2=x^3+kD^3$ of quadratic twists of the curves $y^2=x^3+k$ for certain square-free integer $D$. Given a square-free odd positive integer $D\equiv 2\pmod 3$,  we investigate the families $E_{-D}:~ y^2=x^3-2D^3$ and $E_{3D}:~ y^2=x^3+54D^3$ of the quadratic twists of the curve $y^2=x^3+2$. We prove the following theorem. 
\begin{thm}\label{thm}
Let $D\equiv 2\pmod 3$ be a square-free odd positive integer and $3$ does not divide the class number of $\mathbb{Q}(\sqrt{-2D})$, then the ranks of both $E_{-D}:~ y^2=x^3-2D^3$ and $E_{3D}:~ y^2=x^3+54D^3$ are zero. Assuming Cohen-Lenstra heuristics, there are infinitely many such $D$ for which the ranks of $E_{-D}$ and $E_{3D}$ are zero.
\end{thm}   

\section{Proof of Theorem \ref{thm}}
We begin with the following reflection theorem of Scholz \cite{SC} which is used in the proof. 
\begin{thma}\label{thmsc} Let $D>1$ be a square-free integer. Let $r$ and $s$ be the $3$-ranks of the class groups of the imaginary quadratic field $\mathbb{Q}(\sqrt{-D})$ and the real quadratic field $\mathbb{Q}(\sqrt{3D})$. Then $s\leq r\leq s+1$.
\end{thma}
We denote by $h(d)$ the class number of $\mathbb{Q}(\sqrt{d})$ for any square-free integer $d$. Ankeny, Artin and Chowla \cite[Theorem II]{AAC} (also see, \cite{HS}) proved the following result relating the class numbers of real and imaginary quadratic fields.
\begin{thma}\label{thmB} Let
$d\equiv 1\pmod 3$ be a square-free positive integer and $D= 3d$. Assume that $T$ and $U$ are the coefficients of the fundamental unit of $\mathbb{Q}(\sqrt{D})$. Then $$Th(-d)+Uh(D)\equiv 0\pmod 3.$$
\end{thma}
 We now deduce the following proposition from Theorems \ref{thmsc} and \ref{thmB}. 
 \begin{prop}\label{propcfu}
 Let $D$ be as in Theorem \ref{thm}. Then $3$ does not divide the coefficients of the fundamental unit of $\mathbb{Q}(\sqrt{6D})$. 
\end{prop}

\begin{proof}
Since $D\equiv 2\pmod 3$, so that $2D\equiv 1\pmod 3$. Also, $2D$ is square-free as $D$ is odd and square-free. Therefore by Theorem \ref{thmsc}, $h(6D)\not\equiv 0\pmod 3$ since $h(-2D)\not\equiv 0\pmod 3$.

Let $\varepsilon =T+U\sqrt{6D}$ be the fundamental unit in $\mathbb{Q}(\sqrt{6D})$. Then by Theorem \ref{thmB},
$Th(-2D)+Uh(6D)\equiv 0\pmod 3$.  These together conclude that both $T$ and $U$ are not divisibly by $3$.
\end{proof}
The following Proposition plays crucial role in the proof of Theorem \ref{thm}. Recall that $E_{-D}:~ y^2=x^3-2D^3$ is the quadratic twist of the curve $y^2=x^3+2$. 
\begin{prop}\label{prop1}
Let $D$ be as in Theorem \ref{thm}. Then $$\#\{(x,y)\in E_{-D}(\mathbb{Q}): \text{ord}_p(y)\leq 0,~\text{ for all }  p\mid 6D,~ p\text{ prime}\}=0.$$
\end{prop}

\begin{proof}
To prove this proposition, it is sufficient to show that the equation 
\begin{equation}\label{eq2.1}
y^2=x^3-2D^3z^6
\end{equation}
has no integer solutions in $x,y,z$ with $\gcd(x, y, z)=1$, $\gcd(y,D)=1$ and $z\ne 0$. Without loss of generality, we assume that $(x,y,z)$ is an integer solution of \eqref{eq2.1} such that $y$ and $z$ are positive as well as $z$ is minimal. We can exclude the cases where one or both of $x$ and $y$ are even since these cases would imply that $z$ is even too, which is a contradiction. Thus the only remaining possibility is that both $x$ and $y$ are odd.

Since $D$ is square-free and $\gcd(y,D)=1$, so that $\gcd(x,2D)=1$ and $\gcd(y,2D)=1$. 
We now rewrite \eqref{eq2.1} as follows:
\begin{equation}\label{eq2.2}
\left(y+Dz^3\sqrt{-2D}\right)\left(y-Dz^3\sqrt{-2D}\right)=x^3.
\end{equation} 
Utilizing $\gcd(x,y,z)=1$ and $\gcd(x,2D)=1$, we observe that $\gcd(y+Dz^3\sqrt{-2D}, y-Dz^3\sqrt{-2D})=1$. 
Therefore \eqref{eq2.2} gives 
\begin{equation}\label{eq2.3}
y+Dz^3\sqrt{-2D}=\left(a+b\sqrt{-2D}\right)^3
\end{equation}
for some integers $a$ and $b$ satisfying $\gcd(a,b)=1$ and 
\begin{equation}\label{eq2.4}
a^2+2b^2D=x.
\end{equation}
This shows that $a$ is odd as $x$ is odd. 
Equating the real and imaginary part from \eqref{eq2.3}, we get
\begin{equation}\label{eq2.5}
y=a^3-6ab^2D,
\end{equation}
\begin{equation}\label{eq2.6}
Dz^3=3a^2b-2b^3D.
\end{equation}
Since $D\ne 3$ and it is square-free, so that reading \eqref{eq2.6} modulo $D$, we get $ab\equiv 0\pmod D$. Now \eqref{eq2.4} together with the fact that $\gcd(x,D)=1$ implies $\gcd(a, D)$=1. Thus $b\equiv 0\pmod D$ and we write $b=Db_1$ for some integer $b_1$. Hence \eqref{eq2.6} gives
\begin{equation}\label{eq2.7}
z^3=b_1(3a^2-2b_1^2D^3).
\end{equation}
Reading \eqref{eq2.7} modulo $3$, we get $z\equiv b_1D\pmod 3$.
Utilizing this in \eqref{eq2.7} and then reading modulo $9$, we get
$$b_1^3D^3\equiv 3a^2b_1-2b_1^3D^3\pmod 9.$$
This implies $b_1(a^2-b_1^2D^3)\equiv 0\pmod 3$. If $a^2-b_1^2D^3\equiv 0\pmod 3$, then $1-2\equiv 0 \pmod 3$, since $D\equiv 2\pmod 3$.  Therefore $3\mid b_1$ and thus we write $b_1=3b_2$ for some integer $b_2$. We utilize this in \eqref{eq2.7} to get $3\mid z$, and put $z=3z_1$. Therefore \eqref{eq2.7} takes the form:
\begin{equation}\label{eq2.8}
3z_1^3=b_2(a^2-6b_2^2D^3).
\end{equation}
Since $\gcd(a, b_2)=\gcd(a,b)=1$ and $3\nmid a$, so that \eqref{eq2.8} gives $3\mid b_2$. We put $b_2=3b_3$ for some integer $b_3$. Thus \eqref{eq2.8} becomes
\begin{equation*}%\label{eq2.9}
z_1^3=b_3(a^2-54b_3^2D^3).
\end{equation*}
It is clear that $\gcd(b_3, a^2-54b_3^2D^3)=1$ since $\gcd(a,b_3)=1$ due to $\gcd(a,b)=1$. Thus there exist two integers $A$ and $B$ such that $b_3=B^3$ and $a^2-54b_3^2D^3=A^3$. These together give rise to 
\begin{equation}\label{eq2.9}
a^2-54B^6D^3=A^3.
\end{equation} 
It is clear that $3\nmid A$; otherwise $3\mid a$ which is a contradiction as $3\mid b$ too but $\gcd(a,b)=1$. Since $a$ is odd, so that $A$ is odd too. Also if a prime $p\mid \gcd(a,A)$ then by \eqref{eq2.9}, $p\mid D$ and thus by \eqref{eq2.4} we get $p\mid x$ which contradicts to $\gcd(x, D)=1$. Therefore $\gcd(a, A)=1$.

We now rewrite \eqref{eq2.9} as follows:
$$\left(a+3DB^3\sqrt{6D}\right)\left(a-3DB^3\sqrt{6D}\right)=A^3.$$
It is clear that $\gcd\left(a+3DB^3\sqrt{6D},a-3DB^3\sqrt{6D}\right)=1$ since $\gcd(a,2A)=1$.
Since the class number of $\mathbb{Q}(\sqrt{-2D})$ is not divisible by $3$, so that by 
Theorem \ref{thmsc} the class number of $\mathbb{Q}(\sqrt{6D})$ is not divisible by $3$. Therefore, 
\begin{equation}\label{eq2.10}
a+3DB^3\sqrt{6D}=u\left(\alpha+\beta\sqrt{6D}\right)^3,
\end{equation}
where $u$ is unit in $\mathbb{Q}(\sqrt{6D})$ and $\alpha, \beta$ are integers such that $\gcd(\alpha, \beta)=1$ as $\gcd(a, 3BD)=1$. Since $6D\equiv 2\pmod 4$, so that the fundamental unit, $\varepsilon$ in $\mathbb{Q}(\sqrt{6D})$ is of the form $ \varepsilon=T+U\sqrt{6D}$. Therefore the only possibilities for $u$ are $1, \varepsilon$ and $\varepsilon^2$ since the higher powers of $\varepsilon$ can be  absorbed in $(\alpha+\beta\sqrt{6D})^3$. 

We first assume that $u=1$. Then \eqref{eq2.10} gives:
\begin{equation}\label{eq2.11}
a=\alpha^3+18\alpha\beta^2D,
\end{equation} 
\begin{equation}\label{eq2.12}
DB^3=\alpha^2\beta+2\beta^3 D.
\end{equation}
Since $D$ is square-free, so that \eqref{eq2.12} shows that $D\mid \alpha\beta$. Clearly, $\gcd(D,\alpha)=1$; otherwise by \eqref{eq2.11} $\gcd(D,\alpha)\mid a$ which contradicts to $\gcd(a,D)=1$. Thus $D\mid \beta$, and we put $\beta=D\beta_1$ for some rational integer $\beta_1$. Therefore \eqref{eq2.12} becomes 
$$B^3=\beta_1(\alpha^2+2\beta_1^2D^3).$$
It is clear that $\gcd(\beta_1, \alpha^2+2\beta_1^2D^3)=1$ as $\gcd(\alpha, \beta_1)=1$. Therefore there exist two integers $\beta_2$ and $\alpha_1$ such that $\beta_1=\beta_2^3$ and $\alpha^2+2\beta_1^2D^3=\alpha_1^3$. These further imply 
$$\alpha^2=\alpha_1^3-2D^3\beta_2^6.$$
This shows that $(\alpha_1, \alpha, \beta_2)$ is another solution of $\eqref{eq2.1}$ with $\gcd(\alpha,D)=1$ and $\beta_2\ne 0$. Furthermore, 
$$|\beta_2|=|\beta_1|^{\frac{1}{3}}<|B|=|b_3|^{\frac{1}{3}}=|\frac{b_2}{3}|^{\frac{1}{3}}=|\frac{b_1}{9}|^{\frac{1}{3}}<b_1|^{\frac{1}{3}}.$$
Also, \eqref{eq2.7} gives us  
$|z|=|b_1(3a^2-2b_1^2D^3)|^{\frac{1}{3}}>|b_1|^{\frac{1}{3}}$, and thus $|\beta_2|<|z|$. This contradicts the minimality of $z$. 

We now consider the case where $u=\varepsilon$. Then \eqref{eq2.10} gives the following:
\begin{equation}\label{eq2.13}
a=T\alpha(\alpha^2+18\beta^2D)+18DU\beta(\alpha^2+2D\beta^2),
\end{equation} 
\begin{equation}\label{eq2.14}
3DB^3=3T\beta(\alpha^2+6\beta^2D)+U\alpha(\alpha^2+18D\beta^2).
\end{equation}
We read \eqref{eq2.13} modulo $3$ to get $\alpha^3T\equiv a\pmod 3$. This imples $3\nmid\alpha$  and $3\nmid T$ as $3\nmid a$. Reading \eqref{eq2.14} modulo $3$, we get 
$U\alpha^3\equiv 0\pmod 3$. This implies $3\mid U$ since $3\nmid \alpha$. This contradicts to Proposition \ref{propcfu}.

Finally if $u=\varepsilon^2$, then \eqref{eq2.10} gives 
\begin{equation}\label{eq2.15}
a=(T^2+6DU^2)(\alpha^3+18\alpha\beta^2D)+12DTU(3\alpha^2\beta+6D\beta^3),
\end{equation}
\begin{equation}\label{eq2.16}
3DB^3=(T^2+6DU^2)(3\alpha^2\beta+6D\beta^3)+2TU(\alpha^3+18\alpha\beta^2D).
\end{equation}
Reading \eqref{eq2.15} modulo $3$, we get $3\nmid T$ and $3\nmid \alpha$ since $3\nmid a$. We finally read \eqref{eq2.16} to get $TU\alpha\equiv 0\pmod 3$ which implies $U\equiv 0\pmod 3$. This again contradicts to Proposition \ref{propcfu}. 
\end{proof}

We also need a similar result for the quadratic twist, $E_{3D}:~ y^2=x^3+54D^3$ of the curve $y^2=x^3+2$.  

\begin{prop}\label{prop2}
Let $D$ be as in Theorem \ref{thm}. Then $$\#\{(x,y)\in E_{3D}(\mathbb{Q}):ord_p(y)\leq 0, ~\text{ for all }p\mid 6D,~ p\text{ prime}\}=0.$$
\end{prop}
\begin{proof}
The proof is similar to that of Proposition \ref{prop1}. However, we give the proof for the sake of completeness.  It is sufficient to prove that the equation 
\begin{equation}\label{eq2.17}
y^2=x^3+2(3D)^3z^6
\end{equation}
has no integer solutions in $x,y,z$ with $\gcd(x,y,z)=1$, $\gcd(y,D)=1$ and $z\ne 0$. Without loss of generality, we assume that $(x,y,z)$ is an integer solution of \eqref{eq2.17} such that $y$ and $z$ are positive as well as $z$ is minimal. We can exclude the cases where $\gcd(x,6)\ne1$ or $\gcd(y,6)\ne1$ as these cases would imply that $\gcd(x,y,z)\ne 1$. Thus the only remaining possibility is that both $x$ and $y$ are odd as well as $3\nmid xy$.

Since $D$ is square-free, $\gcd(y,D)=1$ and $\gcd(x,y,z)=1$, so that $\gcd(x,6D)=1$ and $\gcd(y,6D)=1$. 

We now rewrite \eqref{eq2.17} as 
\begin{equation}\label{eq2.18}
(y+3Dz^3\sqrt{6D})(y-3Dz^3\sqrt{6D})=x^3.
\end{equation}
It is clear that $\gcd(y+3Dz^3\sqrt{6D}, y-3Dz^3\sqrt{6D})=1$ as $\gcd(x,y,z)=\gcd(x,2D)=1$. 
Since $3$ does not divide the class number of $\mathbb{Q}(\sqrt{-2D})$, so that by 
Theorem \ref{thmsc}, $3$ does not divide the class number of $\mathbb{Q}(\sqrt{6D})$. Therefore from \eqref{eq2.18} we can write
\begin{equation}\label{eq2.19}
y+3Dz^3\sqrt{6D}=u(a+b\sqrt{6D})^3,
\end{equation}
where $u$ is unit in $\mathbb{Q}(\sqrt{6D})$ and $a, b$ are integers such that $\gcd(a, b)=1$ as $\gcd(y, 3Dz)=1$. Since $6D\equiv 2\pmod 4$, so that the fundamental unit in $\mathbb{Q}(\sqrt{6D})$ is of the form $ T+U\sqrt{6D}$. Therefore $u$ is given by $(T+U\sqrt{6D})^\delta$ with $\delta=0,1,2$ as the higher powers can be absorbed in  $(a+b\sqrt{6D})^3$. 

First consider the case when $\delta=0$. Then \eqref{eq2.19} implies
\begin{equation}\label{eq2.20}
y=a^3+18ab^2D,
\end{equation} 
\begin{equation}\label{eq2.21}
Dz^3=a^2b+2b^3 D.
\end{equation}
The equation \eqref{eq2.21} shows that $D\mid ab$ as $D$ is square-free. It is clear that $\gcd(D,a)=1$; otherwise by \eqref{eq2.20} $\gcd(D,a)\mid y$ which contradicts to $\gcd(y,D)=1$. Hence $D\mid b$, and we write $b=Db_1$ for some integer $b_1$. Thus \eqref{eq2.21} implies 
$$z^3=b_1(a^2+2b_1^2D^3).$$
Since $\gcd(a,b)=1$, so that $\gcd(a, b_1)=1$ and hence $\gcd(b_1, a^2+2b_1^2D^3)=1$. Therefore we can find two integers $B$ and $A$ satisfying $b_1=B^3$ and $a^2+2b_1^2D^3=A^3$. These further give rise to 
$$a^2=A^3-2D^3B^6.$$
This shows that $(A, a, B)$ is another solution of $\eqref{eq2.17}$ satisfying $\gcd(a,D)=1$ and $B\ne 0$. Moreover, 
$$|B|=|b_1|^{\frac{1}{3}}=|z/(a^2+2b_1^2D^3)^{\frac{1}{3}}|<|z|,$$ which contradicts the minimality of $z$. 

We now consider the case where $\delta=1$. In this case, \eqref{eq2.19} gives:
\begin{equation}\label{eq2.22}
y=aT(a^2+18Db^2)+18DUb(a^2+2b^2D), 
\end{equation}
\begin{equation}\label{eq2.23}
3Dz^3=3Tb(a^2+2Db^2)+Ua(a^2+18b^2D). 
\end{equation}
We read \eqref{eq2.22} modulo $3$ to get $a^3T\equiv y\pmod 3$. This imples $3\nmid a$  and $3\nmid T$ as $3\nmid y$. Reading \eqref{eq2.23} modulo $3$, we get 
$Ua^3\equiv 0\pmod 3$. This implies $3\mid U$ since $3\nmid a$ which contradicts to Proposition \ref{propcfu}.

Finally if $\delta=2$, then \eqref{eq2.19} provides 
\begin{equation}\label{eq2.24}
y=(T^2+6DU^2)(a^3+18ab^2D)+36DTUb(a^2+2Db^2),
\end{equation}
\begin{equation}\label{eq2.25}
3Dz^3=(T^2+6DU^2)(3a^2b+6Db^3)+2TUa(a^2+18b^2D).
\end{equation}
Reading \eqref{eq2.24} modulo $3$, we get $3\nmid aT$ since $3\nmid y$. We finally read \eqref{eq2.25} to get $TUa\equiv 0\pmod 3$ which implies $U\equiv 0\pmod 3$. This once again contradicts to Proposition \ref{propcfu}. 
\end{proof}

We also need the following two results in order to complete the proof of Theorem \ref{thm}. The first result is restated from \cite[Ex. 10.19, p. 323]{SI}.

\begin{lem}\label{lsi1} 
For a sixth-power-free integer $m$, let $E(m):~ y^{2}=x^{3}+m$. Then
$E(m)(\mathbb{Q})_{\text{tors}}|6$. More precisely, 
$$
E(m)(\mathbb{Q})_{\text{tors}}\cong \begin{cases}  
\mathbb{Z}/6\mathbb{Z} & \text{ if } m=1,
\\ 
\mathbb{Z}/ 3\mathbb{Z} & \text{ if } m\neq 1 \text{ is a cube}, \text{ or } m=-432,
\\
\mathbb{Z}/ 2\mathbb{Z} & \text{ if } m\neq 1 \text{ is a square},
\\
1& \text{ otherwise}.
\end{cases}
$$ 
\end{lem}
The  following lemma can be deduced from \cite[p. 203]{SI}.
\begin{lem}\label{lsi2} Let $E(m)$ be as in Lemma \ref{lsi1}. Let $P(x,y) \in  E(m)$. Then
$$
\left(x([2]P),~ y([2]P)\right)=\left( \frac{9x^{4}-8y^{2}x}{4y^{2}},~ \frac{-27x^{6}+36y^{2}x^{3}-8y^{4}}{8y^{3}}\right).
$$
\end{lem}

\begin{proof}[Proof of Theorem \ref{thm}]
We assume that $m=-2D^3,~ 54D^3$. Then $E_{-D}:~ y^2=x^3-2D^3$ and $E_{3D}:~ y^2=x^3+54D^3$ can be represented by $E(m)$. Thus utilizing Proposition \ref{prop1} and Proposition \ref{prop2}, we can conclude that
$$
E(m)(\mathbb{Q})=\left\lbrace (x,y) \in \mathbb{Q}^2:~y^2=x^3+m,~\text{ ord}_p(y)\geq 1,~\text{ for all } p\mid 3m, ~p\text{ prime} \right\rbrace. 
$$
Now in order to complete the proof, it suffices to show that $E(m)(\mathbb{Q})$ is finite.
For a prime divisor $p$ of $3m$, $y^2=x^3+m$ provides us the following:
\begin{itemize}
\item[(I)] $\text{ord}_p(y)\geq 1$ if and only if $\text{ord}_p(x)=1$ when $p\ne 3$.
\item[(II)] $\text{ord}_3(y) \geq 1$ if and only if $
\text{ord}_3(x)=\begin{cases}
1 & \text{if } 3\mid m,
\\ 
0 & \text{if } 3\nmid m.
\end{cases}
$ 
\end{itemize}
Applying Lemma \ref{lsi1}, we obtain 
$E(m)(\mathbb{Q})_{tors} = O$. \\
We assume, on the contrary, that $E(m)(\mathbb{Q}) \neq E(m)(\mathbb{Q})_{tors}$. Then we can find $P(x,y) \in E(m)(\mathbb{Q})\setminus E(m)(\mathbb{Q})_{tors}$ and a prime divisor $p$ of $3m$ such that $ord_p(y)\geq 1$. Applying Lemma \ref{lsi2} and utilizing induction on $n$, one gets
$$\text{ord}_p(y([2^{n}]P))\leq 0 \begin{cases} \forall n\geq 1 & \text{ if }p\ne 3,\\
\forall n\geq 2 & \text{ if }p= 3.\end{cases}$$

Assume that $m$ has $t$ distinct primes factors and we put $n=2^{t+1}$.
Then for any prime factor $p$ of $3m$, one gets $\text{ord}_p(y([2^{n}]P)) \leq 0$.
This is a contradiction. 

We can say by the Cohen-Lenstra heuristics \cite{CL} that there exist infinitely many square-free positive integers $D$ such that $D\equiv 2\pmod 3$ and $h(-2D)\not\equiv 0\pmod 3$. For each of these $D$, ranks of $E_{-D}(2)$ and $E_{3D}(2)$ are zero. This complete the proof.     
\end{proof}
\section{Numerical examples}
We compute the ranks of the families of curves, $E_{-D}:~ y^2=x^3-2D^3$ and $E_{3D}:~ y^2=x^3+54D^3$ for square-free positive integers $D\leq 10000$ satisfying the assumptions of Theorem \ref{thm}. However, we list some them in Table \ref{T2}. We use MAGMA \cite{MA} for these computations.

\begin{center}
\tiny
\begin{longtable}{cccc|cccc|cccc}
\caption{Numerical examples of Theorem \ref{thm}. Here, $\rho(E)$ denotes rank of $E$.} \label{T2} \\

\hline \multicolumn{1}{c}{$D$} & \multicolumn{1}{c}{$h(-2D)$} & \multicolumn{1}{c}{$\rho(E_{-D})$}& \multicolumn{1}{c}{$\rho(E_{3D})$} & \multicolumn{1}{|c}{$D$}& \multicolumn{1}{c}{$h(-2D)$} & \multicolumn{1}{c}{$\rho(E_{-D})$} & \multicolumn{1}{c|}{$\rho(E_{3D})$}& \multicolumn{1}{c}{$D$} & \multicolumn{1}{c}{$h(-2D)$}&\multicolumn{1}{c}{$\rho(E_{-D})$} & \multicolumn{1}{c}{$\rho(E_{3D})$}\\ \hline 
\endfirsthead

\multicolumn{12}{c}%
{{\bfseries \tablename\ \thetable{} -- continued from previous page}} \\
\hline \multicolumn{1}{c}{$D$} & \multicolumn{1}{c}{$h(-2D)$} & \multicolumn{1}{c}{$\rho(E_{-D})$}& \multicolumn{1}{c}{$\rho(E_{3D})$} & \multicolumn{1}{|c}{$D$}& \multicolumn{1}{c}{$h(-2D)$} & \multicolumn{1}{c}{$\rho(E_{-D})$} & \multicolumn{1}{c|}{$\rho(E_{3D})$}& \multicolumn{1}{c}{$D$} & \multicolumn{1}{c}{$h(-2D)$}&\multicolumn{1}{c}{$\rho(E_{-D})$} & \multicolumn{1}{c}{$\rho(E_{3D})$}\\ \hline 
\endhead
\hline \multicolumn{12}{r}{{Continued on next page}} \\ \hline
\endfoot
\hline 
\endlastfoot

5 & 2 & 0& 0& 11 &2 &0& 0&17& 4& 0& 0\\
23 &4 &0 &0 &29 &2& 0& 0&35& 4 &0 &0 \\
41& 4& 0 &0 &47 &8& 0& 0&65& 4& 0& 0\\
71 &4& 0& 0&77& 8& 0& 0& 83& 10& 0& 0\\
89 &8 &0& 0&95& 4& 0& 0&113 &8 &0 &0 \\
119 &8& 0& 0&155& 8& 0& 0&161& 8 &0 &0 \\
173 &10& 0& 0&191& 8& 0& 0&197& 10& 0& 0\\
203& 16& 0& 0&209& 8& 0& 0&221 &8 &0 &0 \\
227& 14 &0& 0&233& 8& 0& 0&239 &8 &0& 0\\
251& 14& 0& 0&257& 16& 0& 0 &269& 10& 0& 0\\ 
281& 8& 0& 0&287& 16& 0& 0&299& 8& 0& 0\\
317& 14& 0& 0&323 &16 &0& 0&329& 8 &0 &0 \\
347 &10& 0& 0&371& 8& 0& 0&377 &20 &0 &0 \\
389 &14& 0& 0&395& 16& 0& 0&419& 14& 0& 0\\
431& 8& 0& 0&437& 20& 0& 0&455& 16& 0& 0\\
467& 26& 0& 0&473 &16& 0 &0&479& 16& 0& 0\\
491& 10& 0& 0&497& 16& 0& 0&503& 20& 0& 0\\
527& 16 &0& 0&533& 20& 0& 0&551& 20 &0& 0\\
557& 22& 0& 0&563& 22& 0& 0&623& 16& 0& 0\\
635& 20& 0 &0 &647& 28& 0& 0&659& 10& 0& 0\\
671& 20& 0& 0&677& 22& 0& 0&689 &20& 0& 0\\
695& 20& 0& 0&701& 14& 0& 0&707& 28& 0& 0\\
713 &32& 0& 0&719 &16& 0& 0& 731& 16& 0& 0\\
737& 16& 0& 0&743& 20& 0& 0&749 &16 &0 &0\\
755& 16& 0& 0&761& 20& 0& 0&767& 20& 0& 0\\
773& 34& 0& 0&785& 20& 0 &0 &791& 16& 0& 0\\
797& 34& 0& 0&803& 28& 0& 0&815& 28& 0& 0\\
821 &14& 0& 0& 827& 22& 0& 0&839& 20& 0 &0\\
851 &20& 0 &0&857& 20& 0 &0 &869& 16& 0 &0\\
887 &20& 0& 0& 893& 20& 0& 0 &899& 20& 0 &0 \\
905& 20& 0 &0 &911 &16& 0& 0& 923 &28& 0 &0 \\
929& 20& 0& 0&935& 16 &0 &0&953& 20& 0& 0\\
959& 16& 0 &0& 965& 20& 0& 0&971& 22& 0 &0\\
977& 28& 0& 0& 1001& 16& 0& 0&1013& 34& 0& 0
%1037 20 0 0 0 0
%1055 28 0 0 0 0
%1067 32 0 0 0 0
%1073 28 0 0 0 0
%1079 20 0 0 0 0
%1085 16 0 0 0 0
%1091 22 0 0 0 0
%1097 32 0 0 0 0
%1103 32 0 0 0 0
%1109 26 0 0 0 0
%1115 20 0 0 0 0
%1121 16 0 0 0 0
%1145 20 0 0 0 0
%1151 16 0 0 0 0
%1157 28 0 0 0 0
%1163 22 0 0 0 0
%1181 26 0 0 0 0
%1187 22 0 0 0 0
%1193 32 0 0 0 0
%1199 20 0 0 0 0
%1205 16 0 0 0 0
%1211 20 0 0 0 0
%1217 32 0 0 0 0
%1223 28 0 0 0 0
%1235 32 0 0 0 0
%1247 20 0 0 0 0
%1253 32 0 0 0 0
%1271 16 0 0 0 0
%1277 38 0 0 0 0
%1289 16 0 0 0 0
%1301 26 0 0 0 0
%1307 46 0 0 0 0
%1319 20 0 0 0 0
%1337 32 0 0 0 0
%1343 32 0 0 0 0
%1361 20 0 0 0 0
%1367 28 0 0 0 0
%1373 38 0 0 0 0
%1391 28 0 0 0 0
%1409 20 0 0 0 0
%1415 20 0 0 0 0
%1427 26 0 0 0 0
%1433 40 0 0 0 0
%1439 16 0 0 0 0
%1463 40 0 0 0 0
%1481 16 0 0 0 0
%1505 32 0 0 0 0
%1511 28 0 0 0 0
%1517 40 0 0 0 0
%1523 34 0 0 0 0
%1529 32 0 0 0 0
%1535 20 0 0 0 0
%1541 20 0 0 0 0
%1553 32 0 0 0 0
%1559 20 0 0 0 0
%1565 28 0 0 0 0
%1577 32 0 0 0 0
%1583 32 0 0 0 0
%1589 20 0 0 0 0
%1607 44 0 0 0 0
%1613 38 0 0 0 0
%1631 32 0 0 0 0
%1643 52 0 0 0 0
%1649 16 0 0 0 0
%1661 16 0 0 0 0
%1673 40 0 0 0 0
%1679 32 0 0 0 0
%1691 20 0 0 0 0
%1697 52 0 0 0 0
%1703 28 0 0 0 0
%1709 22 0 0 0 0
%1721 40 0 0 0 0
%1727 44 0 0 0 0
%1733 26 0 0 0 0
%1745 20 0 0 0 0
%1751 16 0 0 0 0
%1769 28 0 0 0 0
%1781 28 0 0 0 0
%1787 38 0 0 0 0
%1793 32 0 0 0 0
%1811 26 0 0 0 0
%1823 32 0 0 0 0
%1841 32 0 0 0 0
%1853 52 0 0 0 0
%1865 28 0 0 0 0
%1877 34 0 0 0 0
%1901 22 0 0 0 0
%1907 46 0 0 0 0
%1913 32 0 0 0 0
%1919 28 0 0 0 0
%1931 26 0 0 0 0
%1937 52 0 0 0 0
%1949 26 0 0 0 0
%1961 20 0 0 0 0
%1967 32 0 0 0 0
%1973 50 0 0 0 0
%1979 22 0 0 0 0
%1985 28 0 0 0 0
%1991 28 0 0 0 0
 \end{longtable}
\end{center}

\section*{Acknowledgements}
The authors are grateful to Prof. Kalyan Chakraborty for his careful reading, helpful comments and suggestions.  The authors are grateful to the anonymous referee(s) for valuable comments, which have improved the presentation of the paper. This work was supported by SERB CRG grant (No. CRG/2023/007323) and MATRICS grant (No. MTR/2021/000762), Govt. of India.

%\subsection*{Data availability statement} This manuscript has no associate data.

\end{document}